\documentclass{amsart}

\usepackage{graphicx}
\usepackage[latin1]{inputenc}
\newtheorem{thm}{Theorem}[section]
\newtheorem{cor}[thm]{Corollary}
\newtheorem{lem}[thm]{Lemma}
\newtheorem{prop}[thm]{Proposition}
\newtheorem{exa}[thm]{Example}
\theoremstyle{definition}
\newtheorem{dfn}[thm]{Definition}
\theoremstyle{remark}

\numberwithin{equation}{section}
\begin{document}

\title[]{The group $G_{n}^{2}$ with a parity and with points}%
\author{S.Kim}

\address{S.Kim, Bauman Moscow State Technical University,}%
\email{ksj19891120@gmail.com}%
%\subjclass{57M25}%

%\date{\today}%
%\dedicatory{}%
%\commby{}%
% ----------------------------------------------------------------

\maketitle

\begin{abstract}
In~\cite{Ma} Manturov studied groups $G_{n}^{k}$ for fixed integers $n$ and $k$ such that $k<n$. In particular, $G_{n}^{2}$ is isomorphic to the group of free braids of $n$-stands. In~\cite{KiMa} Manturov and the author studied an invariant valued in free groups not only for free braids but also for free tangles, which is derived from the group $G_{n}^{2}$. On the other hands, in~\cite{FeMa} Manturov and Fedoseev studied groups $Br_{2}^{n}$ of virtual braids with parity and groups $Br_{d}^{n}$ of virtual braids with dots. They showed that there is homomorphism from $Br_{2}^{n}$ to  $Br_{d}^{n}$ and proved the following statement:

{\it if two braids with parity are equivalent as braids with dots, then they are equivalent as braids with parity.}

By the statement it is deduced that a parity of the braid can be represented by a geometric object, dots on strands. 

In this paper we study $G_{n}^{2}$ with structures, which are corresponded to parity and points on a braid, which are denoted by $G_{n,p}^{2}$ and $G_{n,d}^{2}$, respectively. In section 3, it is proved that there is a monomorphism from $G_{n}^{2}$ to $G_{n,p}^{2}$ and that there is a monomorphism from $G_{n,p}^{2}$ to $G_{n,d}^{2}$. By the homomorphism  from $G_{n}^{2}$ to $G_{n,p}^{2}$, it can be deduced that a given parity of a braid has geometric representation, which is the number of points on the braid. In section 4, it can be proved that for each element $\beta$ in $G_{n,d}^{2}$, an element in $G_{n+1}^{2}$ obtained by adding another strand by tracing points on $\beta$. That is, a parity of free braid of $n$-stands is represented not only by points on strands, but also by an $n+1$-th strand. Conversely, for a braid of $n+1$-strands, a braid of $n$-strands is obtained by deleting one strand of the braid of $n+1$-strand. Finally, we will simply discuss about the way to adjust the previous observations to know whether a given braid is Brunnian or not. 
\end{abstract}

\section{Introduction}

In~\cite{Ma}, Manturov studied groups $G_{n}^{k}$ given by the group presentation, from which many invariants for knots and dynamical systems are derived. In~\cite{MaNi}, V.O.Manturov and I.M.Nikonov found that pure braids groups are embedded in $G_{n}^{k}$ and constructed invariants valued in free product of cyclic groups. Tangles are significant generalization of braids and we have the following question : is it possible to generalize invariants, which are asserted, to the case of tangles. By $G_{n}^{2}$ analogously tangles can be studied. In~\cite{KiMa}, the invariants valued in free groups are extended to the case of free tangles. The groups $G_{n}^{2}$ which are simplest case of the $G_{n}^{k}$ are known as free braid groups of $n$-strands, which are known as other names, see~\cite{Ba} and \cite{BaBeDa}.

On the other hands, in~\cite{FeMa} V.O.Manturov and D.A.Fedoseev studied virtual braids group $Br_{2}^{n}$ with parity and virtual braids group $Br_{d}^{n}$ with dots. They showed that there is a homomorphism $h$ from $Br_{2}^{n}$ to $Br_{d}^{2}$. And it was shown that $Br_{2}^{n}$ is homomorphically embedded in $Br_{d}^{n}$ by showing the following statement:

{\it if two braids with parity are equivalent as braids with dots, then they are equivalent as braids with parity.}

It means that the parity of braids has geometric meaning; the number of points on braids.

In this paper,  analogously we will generalizethe groups $G_{n}^{2}$ with structures; parity and points, which are denoted by $G_{n,p}^{2}$ and $G_{n,d}^{2}$, respectively. We are interested in the following question:\\
{\it two objects are equivalent in a lager category, then are they equivalent as themselves?}\\
Which means that a category is embedded in the larger category. For example, the following statements are related with the above problem.
\begin{center}
\begin{enumerate}
\item if two classical knots are equivalent as virtual knots, then they are equivalent as classical knots.
\item if two classical braids are equivalent as virtual briads, then they are equivalent as classical briads.
\end{enumerate}
\end{center}
In~\cite{GoPoVi} the first statement is proved and the second statement is proved in~\cite{FeRiRo}.

As the above question, we showed the following two statements:
\begin{center}
\begin{enumerate}
\item if two element in $G_{n}^{2}$ are equivalent in $G_{n,p}^{2}$, then they are equivalent in $G_{n}^{2}$.
\item if two element in $G_{n,p}^{2}$ are equivalent in $G_{n,d}^{2}$, then they are equivalent in $G_{n,p}^{2}$.
\end{enumerate}
\end{center}

From the above statements, the followings are deduced: Since $G_{n}^{2}$ is isomorphic to group of free braids with $n$ strands, parity can be generalized to the case of free braids with respect to $G_{n,p}^{2}$. Moreover, parity is defined abstractly, but it can be geometrically presented by points.
On the other hands point on a braid of $n$ strands means 'liking' between $(n+1)$-th strand and others. Moreover points on the braid of $n$-strands can be considered as {\it traces} of $n+1$-th strand and a homomorphism from $G_{n,d}^{2}$ to $G_{n+1}^{2}$ can be defined by adding a new strand by `following the traces'(points) of it. By the above observations, we can find a homomorphism from $G_{n,p}^{2}$ to $G_{n+1}^{2}$ and, roughly speaking, each value of parity of crossings of the braid of $n$-strands can be represented by the number of `linking' between $n+1$-strand and the other two strands, in which a crossing is.
Conversely a braid with points can be obtained from a braid of $n+1$ strands by deleting one strand and placing points on places, on which deleted stand was.And a braid with parity can be obtained from a braid with points by counting the number of points. Then two braids can be distinguished by braids with parity.
In section 2, we simply remind basic definitions and simple results. In section 3, we consider the following question:\\
{\it two objects are equivalent in a larger category, then are they equivalent as themselves?}\\
in the case of $G_{n}^{2}$, $G_{n,p}$ and $G_{n,d}^{2}$. In section 4, relations between $G_{n,d}^{2}$ and $G_{n+1}$ are studied. And we will show geometric meanings and examples of braids of $n+1$ strand, which are distinguished from trivial braids by braids with parity.

\section{Basic definitions}
Firstly, we recall basic definitions and introduce simple results.
\begin{dfn}\cite{Ma2}
Let $G_{n}^{2}$ be the group given by the presentation generated by $\{ a_{ij}~|~ \{i,j\} \subset \{1, \dots, n\}, i < j \}$ with the following relations:

\begin{enumerate}
\item $a_{ij}^{2} = 1$ for all $i \neq j$, 
\item $a_{ij}a_{kl} = a_{kl}a_{ij}$ for distinct $i,j,k,l$,
\item $a_{ij}a_{ik}a_{jk} = a_{jk}a_{ik}a_{ij}$ for distinct $i,j,k$.
\end{enumerate}
\end{dfn}

In~\cite{Ma2} Manturov showed that {\it colored (pure) free braids} can be presented by elements in $G_{n}^{2}$. In fact, the relations for $G_{n}^{2}$ are related with Reidemeister moves for free braids.
\begin{dfn}
A {\it free braid diagram} is a graph inside a rectangle $\mathbb{R} \times [0,1]$ with the following properties:
\begin{enumerate}
\item The graph vertices of valency $1$ are the points $[i,0]$ and $[i,1]$, $i = 1, \cdots, n$.
\item All other  vertices are $4$-valent vertices. 
\item For each $4$-valent vertex, edges are split by two pairs, two edges in which are called {\it opposite edeges}. The opposite edges in such vertices are at the angle of $\pi$.
\item The braid strands monotonously go down.
\end{enumerate}
For a free braid diagram, if $[i,0]$ and $[i,1]$ are included in the same strand for every $i \in \{1, \cdots n\}$, then the diagram is called {\it a pure free braid diagram}.
\end{dfn}
A {\it colored free braid diagram} is a free braid diagram such that each component is enumerated. If it is a pure free braid diagram, we may assume that the enumeration agrees with $x$-terms of $1$-valent vertices of the strands.

\begin{dfn}
{\it A free braid of $n$ strand} is an equivalence class of free diagrams of n-strand braid under Artin moves for free braids in Fig.~\ref{movesArtin}.

\begin{figure}[h!]
\begin{center}
 \includegraphics[width =12cm]{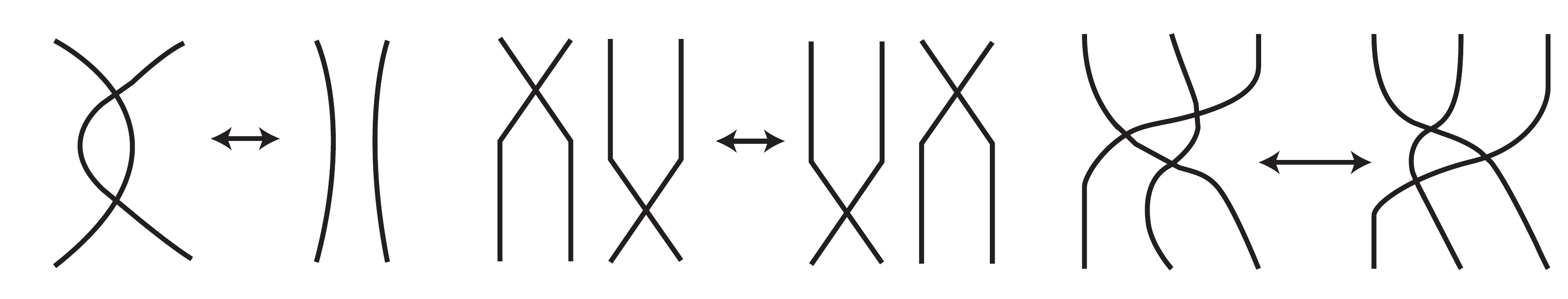}

\end{center}

 \caption{Artin moves for free braids}\label{movesArtin}
\end{figure}
\end{dfn}

This moves are related with relations of $G_{n}^{2}$ in order. Now we construct a function from the set $PB_{n}$ of colored pure free braids with $n$ components to $G_{n}^{2}$. Let $\{\epsilon_{k}\}_{k=1}^{m} \subset (0,1)$ be a set of values such that a $4$-valent vertex is placed in $\mathbb{R} \times \{\epsilon_{k}\}$ and $\epsilon_{k} < \epsilon_{k+1}$ for each $k = 1, \cdots m-1$. Without loss of generality, we assume that there is the only one vertex on $\mathbb{R} \times \{\epsilon_{k}\}$ for each $k$. Let $c_{k}$ be the $4$-valent vertex on $\mathbb{R} \times \{\epsilon_{k}\} $. Define a function $\iota$ from $PB_{n}$ to $G_{n}^{2}$ as the following: For each vertex $c_{k}$, if it contains edges in $i_{k}$-th and $j_{k}$-th components, then give a generator $a_{i_{k}j_{k}}$ in $G_{n}^{2}$. Define $\iota(\beta) = a_{i_{1}j_{1}} \cdots a_{i_{m}j_{m}}$, see Fig.~\ref{exa-freebraidword}.

%We construct a function from the set $PB_{n}$ of colored pure free braids with $n$-components to $G_{n}^{2}$. Assume that if a $4$-valent vertex in on $\mathbb{R} \times \{\epsilon\} $ for some $\epsilon \in (0,1)$, 
%there are no other vertices. Define a function $\iota$ from $PB_{n}$ to $G_{n}^{2}$ as the followings; for a colored pure free braids $\beta$, there are finitely many $\{\epsilon_{k}\}_{k=1}^{m} \subset (0,1)$ such that  $\epsilon_{k} < \epsilon_{k+1}$ for each $k = 1, \cdots m-1$ and there is a 4-valent vertex $c_{k}$ on $\mathbb{R} \times \{ \epsilon_{k}\}$. Note that each $4$-valent vertex is corresponded with two different strands. For each vertex $c_{k}$, if it contains edges in $i_{k}$-th and $j_{k}$-th components, then give a generator $a_{i_{k}j_{k}}$ in $G_{n}^{2}$. Define $\iota(\beta) = a_{i_{1}j_{1}} \cdots a_{i_{m}j_{m}}$, see Fig.~\ref{exa-freebraidword}.

\begin{figure}[h!]
\begin{center}
 \includegraphics[width =8cm]{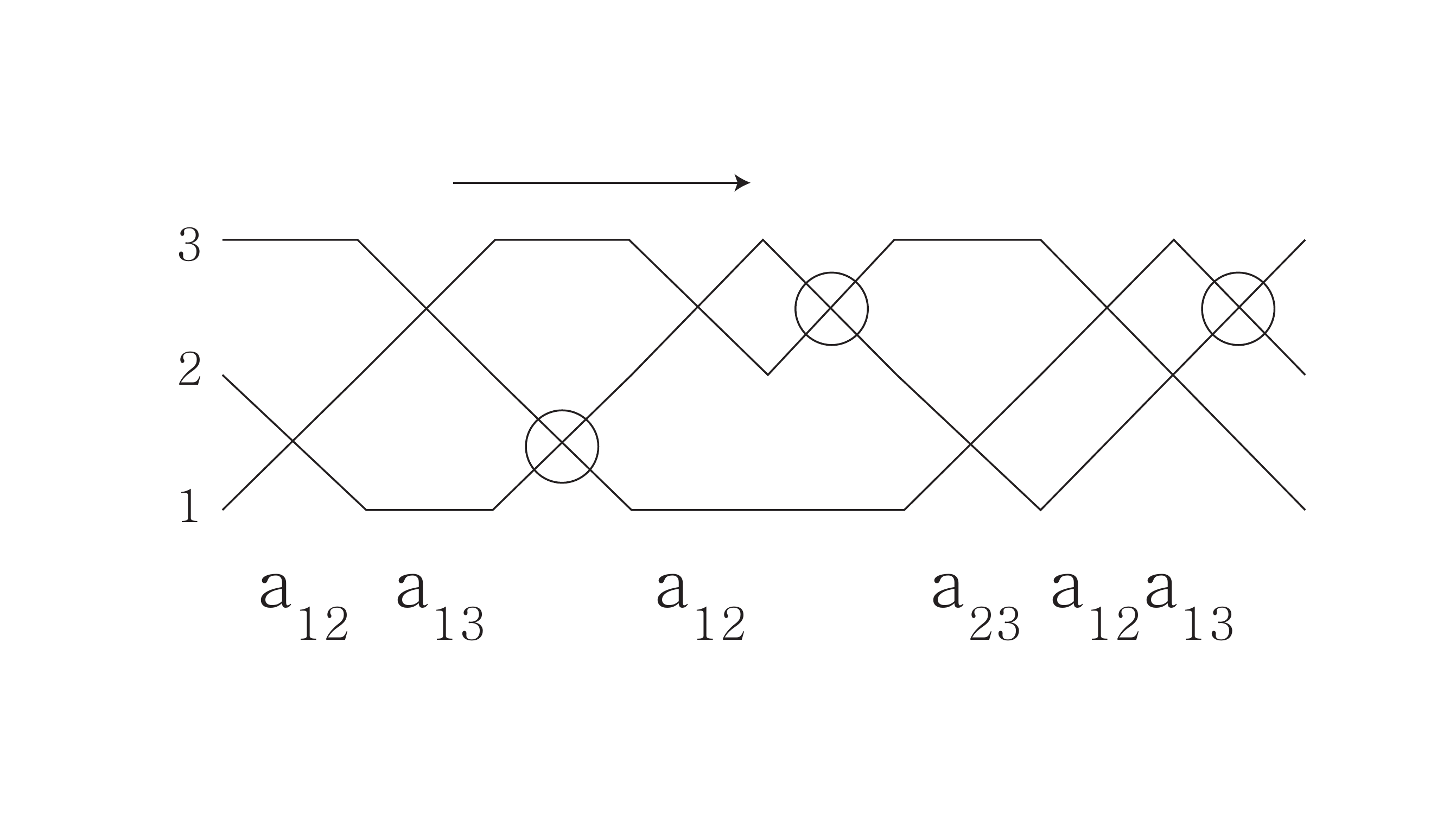}

\end{center}

 \caption{Example of a word for a free braid}\label{exa-freebraidword}
\end{figure}

\begin{prop}\cite{Ma2}
A function $\iota$ from $PB_{n}$ to $G_{n}^{2}$ is well-defined.
\end{prop}

From the discovery of parity by Manturov~\cite{Ma1}, invariants for classical knots are extended to invariants for virtual knots. In~\cite{FeMa} Fedoseev and Manturov studied classical braids with structures, which are related to parity and dots on strands of classical braids. With those structures, parity can be extended to the case of braids and generalized parity can be deduced by braids with dots. The group of braids with parity and with points are defined as follows:
 
\begin{dfn}\cite{FeMa}
{\it A group of $n$-braids with parity} $Br_{n}^{2}$ is a group generated by $\{\sigma_{i}^{0}, \sigma_{i}^{1}\}_{i=1}^{n-1}$ with the following relations;

\begin{enumerate}

\item $\sigma^{\epsilon_{1}}_{i} \sigma^{\epsilon_{2}}_{j} = \sigma^{\epsilon_{2}}_{j} \sigma^{\epsilon_{1}}_{i}$ for $|i-j|>2$,
\item $\sigma_{i}^{\epsilon_{1}}\sigma_{i+1}^{\epsilon_{2}}\sigma_{i}^{\epsilon_{3}} = \sigma_{i+1}^{\epsilon_{3}}\sigma_{i}^{\epsilon_{2}}\sigma_{i+1}^{\epsilon_{1}}$, where $\epsilon_{1}+ \epsilon_{2}+\epsilon_{3} = 0$ mod $2$.
\end{enumerate}
An element in $Br_{2}^{n}$ is called {\it a $\mathbb{Z}_{2}$-braid}.
\end{dfn}

\begin{dfn}\cite{FeMa}
{\it A group of $n$-braids with dots} $Br_{n}^{2}$ is a group generated by $\{ \sigma_{i} ~|~ i \in \{1, \cdots, n-1\} \}, \{ \tau_{i} ~|~ i \in \{1, \cdots, n\} \}$ with the following relations;
\begin{enumerate}
\item $\sigma_{i}^{2} = 1$ for all $i \in \{1, \cdots n-1\}$,
\item $\sigma_{i}\sigma_{j} = \sigma_{j}\sigma_{i}$ for $|i-j| \geq2$,
\item $\sigma_{i}\sigma_{i+1}\sigma_{i} = \sigma_{i+1}\sigma_{i}\sigma_{i+1}$ for $1 \leq i \leq n-1$,
\item $\tau_{i}^{2} = 1$ for all $i \in \{1, \cdots n\}$,
\item $\tau_{i}\tau_{j} = \tau_{j}\tau_{i}$ for all $i,j \in \{1, \cdots n\}$,
\item $\tau_{i}\tau_{i+1}\sigma_{i}\tau_{i+1}\tau_{i} = \sigma_{i}$ for all $i \in \{1, \cdots n-1\}$,
\item $\sigma_{i}\tau_{k} = \tau_{k}\sigma_{i}$ for $|i-k| \geq2$.
\end{enumerate}
\end{dfn}

In the group $Br_{n}^{d}$, each generator $\tau_{i}$ is presented by a point on $i$-th strand of the braid. 
Now we define a mapping $f : Br_{n}^{2} \rightarrow Br_{n}^{d} $ by 
\begin{center}
$f(\sigma_{i}^{\epsilon})  = \left\{
\begin{array}{cc} % brackets may be (...), [...], \{...\}, or left out
     \sigma_{i} & \text{if}~\epsilon = 0, \\
      \tau_{i} \sigma_{i}\tau_{i}
 &  \text{if}~\epsilon= 1. \\
   \end{array}\right.$
   \end{center}
 In~\cite{FeMa} they proved that this mapping is injective homomorphism. Moreover the following proposition can be proved.

%By the relation (4), two consecutive points can be deleted, that is, the number of $\tau$'s is changed by modulo $2$. From this observation, it can be shown that a function $f : Br_{2}^{n} \rightarrow Br_{d}^{n} $ defined by 
%\begin{center}
%$f(\sigma_{i}^{\epsilon})  = \left\{
%\begin{array}{cc} % brackets may be (...), [...], \{...\}, or left out
   %  \sigma_{i} & \text{if}~\epsilon = 0 \\
  %    \tau_{i} \sigma_{i}\tau_{i}
 %&  \text{if}~\epsilon= 1. \\
 %  \end{array}\right.$
%   \end{center}
%   is well-defined and, moreover, it is injective. And the following proposition is followed. 

\begin{prop}
If two $\mathbb{Z}_{2}-$braids are equal as braids with dots, they are equal as $\mathbb{Z}_{2}-$braids.
\end{prop}
%Here comment proposition "if two diagram are equi as..."
In the following sections, we study $G_{n}^{2}$ with structures which are related to parity and points algebraically.

\section{$G_{n}^{2}$ with additional structures}

In this section, we get enhancements of $G_{n}^{2}$ by adding some structures, which are related to parity and points on strands of colored pure braids. Firstly, we define a group $G_{n,p}^{2}$ and we call it {\it $G_{n}^{2}$ with parity.}

\begin{dfn}
For a positive integer $n$, define $G_{n,p}^{2}$ by the group presentation generated by $ \{ a_{ij}^{\epsilon} ~|~ \{i,j\} \subset \{1, \dots, n\}, i < j, ~\epsilon \in \{0,1\}  \}$ with the following relations;
\begin{center}\begin{enumerate}
\item $(a_{ij}^{\epsilon})^{2} = 1$, $\epsilon \in \{0,1\}$ and $i,j \in \{1, \cdots, n\}$,
\item $a_{ij}^{\epsilon_{ij}}a_{kl}^{\epsilon_{kl}} = a_{kl}^{\epsilon_{kl}}a_{ij}^{\epsilon_{ij}}$ for $1 \leq i <j<k<l \leq n$,
\item $a_{ij}^{\epsilon_{ij}}a_{ik}^{\epsilon_{ik}}a_{jk}^{\epsilon_{jk}} = a_{jk}^{\epsilon_{jk}}a_{ik}^{\epsilon_{ik}}a_{ij}^{\epsilon_{ij}}$,  for $1 \leq i <j<k \leq n$, where $\epsilon_{ij}+\epsilon_{ik}+\epsilon_{jk}=0$ mod $2$.
\end{enumerate}
\end{center}
\end{dfn}

The relations of $G_{n,p}^{2}$ are closely related to the axioms of parity \cite{Ma1}. We call $a_{ij}^{0}$($a_{ij}^{1}$) {\it an even generator}({\it an odd generator}). If a word $\beta$ in $G_{n,p}^{2}$ has no odd generators, then we call $\beta$ {\it an even word} (or {\it an even element}). 
$G_{n}^{2}$ can be considered as a subgroup of $G_{n,p}^{2}$ by a homomorphism $i$ from $G_{n}^{2}$ to $G_{n,p}^{2}$ defined by $i(a_{ij}) = a_{ij}^{0}$. It is easy to show that $i$ is well-defined. 
%This map is closely related with the following fact:
%{\it every crossing of a classical knot has the value $0$ of the Gaussian parity.}\\
Moreover, the following statement can be proved.

\begin{lem}\label{emb_to_parity}%This lemma should be asserted in introduction
The homomorphism $i : G_{n}^{2} \rightarrow G_{n,p}^{2}$ is a monomorphism.
\end{lem}

\begin{proof}
To prove this statement, consider projection map $p : G_{n,p}^{2} \rightarrow G_{n}^{2}$ defined by 
\begin{center}
$p(a_{ij}^{\epsilon})  = \left\{
\begin{array}{cc} 
     a_{ij} & \text{if}~\epsilon = 0, \\
       1
 &  \text{if}~\epsilon= 1. \\
   \end{array}\right.$
   \end{center}
It is easy to show that $p$ is well-defined function. Let $\beta$ and $\beta'$ be two words in $G_{n}^{2}$ such that $i(\beta)= i(\beta')$. By the definition of $i$, $p(i(\beta)) =\beta$ and  $p(i(\beta')) =\beta'$. That is, $\beta=\beta'$ in $G_{n}^{2}$, therefore, the proof is completed.

\end{proof}

\begin{cor}
If two words $\beta$ and $\beta'$ in $G_{n}^{2}$ are equivalent in $G_{n,p}^{2}$, then they are equivalent in $G_{n}^{2}$.
\end{cor}

Now, we define a group $G_{n,d}^{2}$ and we call it $G_{n}^{2}$ {\it with points.}
\begin{dfn}
For a positive integer $n$, define $G_{n,d}^{2}$ by the group presentation generated by  $\{ a_{ij} ~|~ \{i,j\} \subset \{1, \dots, n\}, i < j\}$ and $\{ \tau_{i} ~|~ i \in \{1, \cdots, n\} \}$ with the following relations;
\begin{enumerate}
\item $a_{ij}^{2} = 1$ for $\{i,j\} \subset \{1, \dots, n\}, i < j$,
\item $a_{ij}a_{kl} = a_{kl}a_{ij}$ for distinct $i,j,k,l \in \{1, \dots, n\}$,
\item $a_{ij}a_{ik}a_{jk} = a_{jk}a_{ik}a_{ij}$ for distinct $i,j,k \in \{1, \dots, n\}$
\item $\tau_{i}^{2} = 1$ for $i \in \{1, \dots, n\}$,
\item $\tau_{i}\tau_{j} = \tau_{j}\tau_{i}$ for $i,j\in \{1, \dots, n\}$,
\item $\tau_{i}\tau_{j}a_{ij}\tau_{j}\tau_{i} = a_{ij}$ for $i,j \in \{1, \dots, n\}$,
\item $a_{ij}\tau_{k} = \tau_{k}a_{ij}$ for distinct $i,j,k \in \{1, \dots, n\}$.
\end{enumerate}
\end{dfn}

We call $\tau_{i}$ {\it a generator for a point on $i$-th component} or simply, {\it a point on  $i$-th component.} Geometrically, $a_{ij}$ is corresponded to a 4-valent vertex and $\tau_{i}$ is corresponded to a point on the $i$-th strand of the free braid, see Fig.~\ref{exa-cro_point}.

\begin{figure}[h!]
\begin{center}
 \includegraphics[width =8cm]{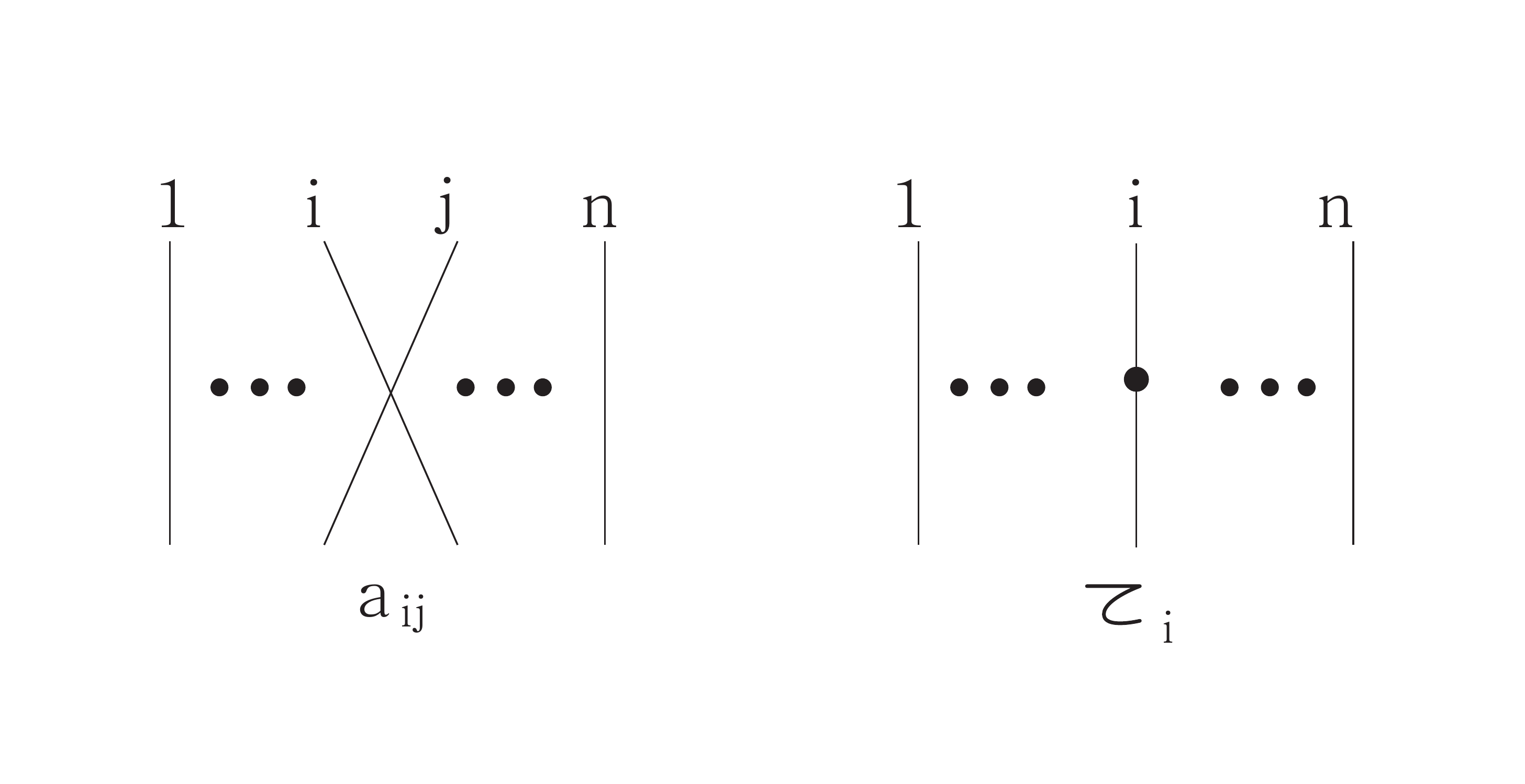}

\end{center}

 \caption{Geometrical meanigngs of $a_{ij}$ and $\tau_{i}$}\label{exa-cro_point}
\end{figure}

In fact, two groups $G_{n,p}^{2}$ and $G_{n,d}^{2}$ are closely related to each other.
Define a function $\phi$ from $G_{n,p}^{2}$ to $G_{n,d}^{2}$ by
\begin{center}
$\phi(a_{ij}^{\epsilon}) = \left\{
\begin{array}{cc} 
     a_{ij} & \text{if}~\epsilon =0, \\
    \tau_{i}a_{ij}\tau_{i}
 &  \text{if}~\epsilon =1. \\
   \end{array}\right.$
\end{center}
Geometrically, the image of $a_{ij}^{1}$ is a crossing between $i$-th and $j$-th strands of the braid with two points just before and after of the crossing on $i$-th strand, see Fig.~\ref{imageofphi}. 
\begin{figure}[h!]
\begin{center}
 \includegraphics[width =10cm]{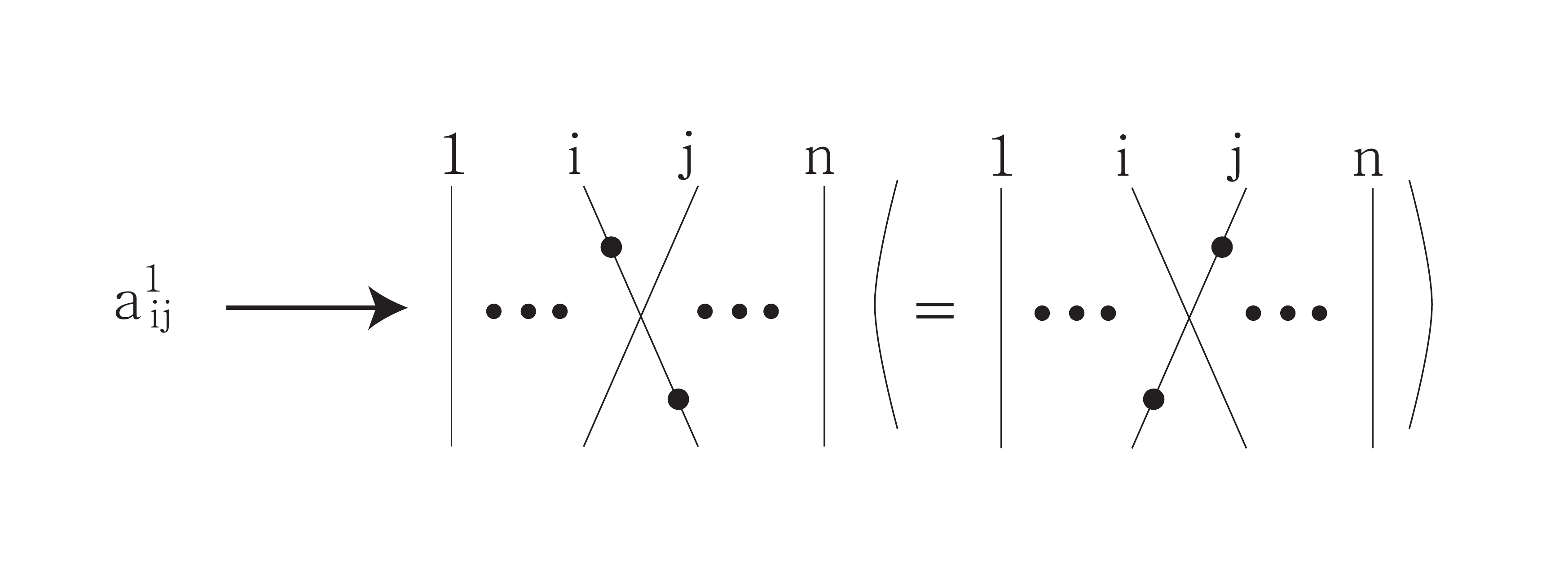}

\end{center}

 \caption{Image of $a_{ij}^{1}$}\label{imageofphi}
\end{figure}
Notice that, since we can get $\tau_{i}a_{ij}\tau_{i}=\tau_{j}a_{ij}\tau_{j}$ from the relation $\tau_{i}\tau_{j}a_{ij}\tau_{j}\tau_{i} =a_{ij}$, two diagrams with point in Fig.~\ref{imageofphi} are equivalent. On the other hand, the number of $\tau_{i}$ and $\tau_{j}$ before a given crossing $\phi(a_{ij}^{\epsilon})$ is equal to $\epsilon$ modulo $2$, because every image of $a_{kl}^{\epsilon}$ before $\phi(a_{ij}^{\epsilon})$ has two points or no points. 
\begin{lem}%By figures prove this lemma
The map $\phi$ from $G_{n,p}^{2}$ to $G_{n,d}^{2}$ is well defined.
\end{lem}

\begin{proof}
To show that $\phi$ is well defined, it is enough to show that every relation is preserved by $\phi$. For relations $(a_{ij}^{\epsilon})^{2} = 1$, if $\epsilon =0$, then $\phi((a_{ij}^{0})^{2}) = a_{ij}^{2} = 1$. If $\epsilon =1$, then  $$\phi((a_{ij}^{1})^{2}) = \tau_{i}a_{ij}\tau_{i}\tau_{i}a_{ij}\tau_{i} = \tau_{i}a_{ij}a_{ij}\tau_{i} =  \tau_{i}\tau_{i} = 1,$$
which can be presented geometrically as Fig.~\ref{proof-lem-A2}.
\begin{figure}[h!]
\begin{center}
 \includegraphics[width =12cm]{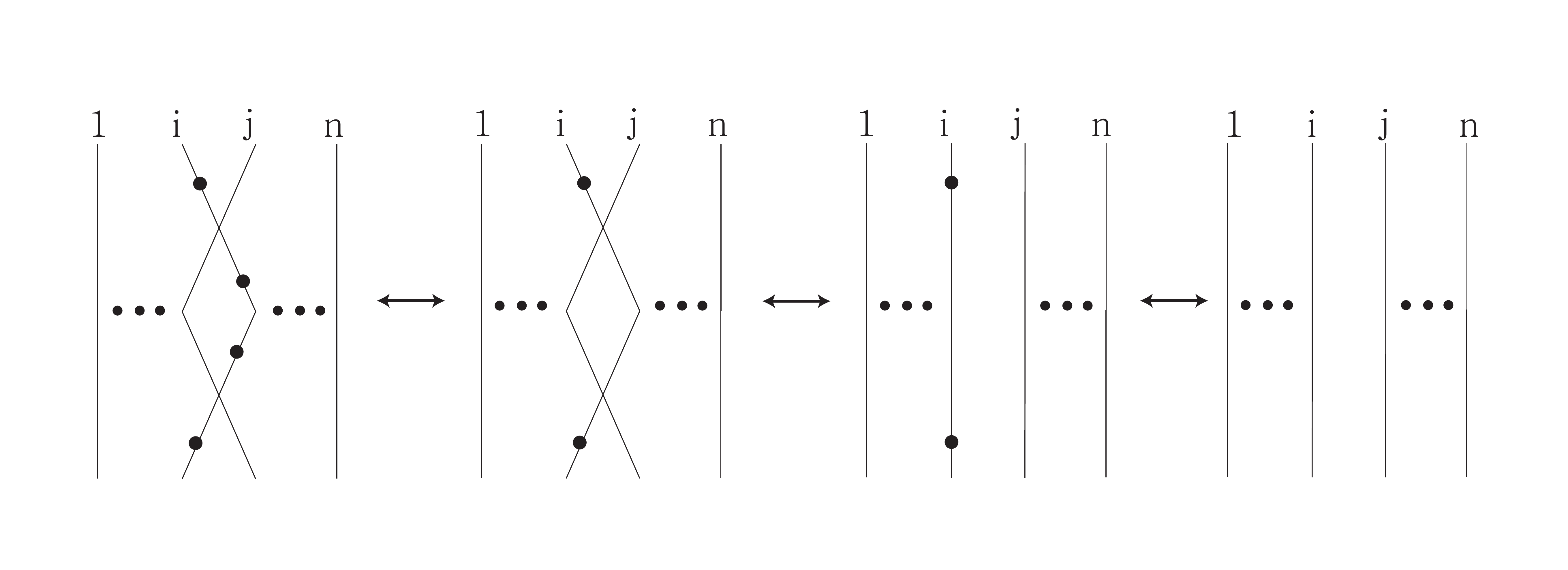}

\end{center}

 \caption{$\phi(a_{ij}^{2}) =1$}\label{proof-lem-A2}
\end{figure}
For relations $a_{ij}^{\epsilon_{ij}}a_{kl}^{\epsilon_{kl}} = a_{kl}^{\epsilon_{kl}}a_{ij}^{\epsilon_{ij}}$ with distinct $i,j,k,l$, since $i$,$j$,$k$, and $l$ are different indices, clearly the commutativity holds. For relations $a_{ij}^{\epsilon_{ij}}a_{ik}^{\epsilon_{ik}}a_{jk}^{\epsilon_{jk}} = a_{jk}^{\epsilon_{jk}}a_{ik}^{\epsilon_{ik}}a_{ij}^{\epsilon_{ij}}$, where $\epsilon_{ij}+\epsilon_{ik}+\epsilon_{jk}=0$ mod $2$, there are only two cases: all $\epsilon$'s are equal to $0$ or only two of them are equal to $1$. If every $\epsilon$ is $0$, then 
$$\phi(a_{ij}^{0}a_{ik}^{0}a_{jk}^{0}) = a_{ij}a_{ik}a_{jk} = a_{jk}a_{ik}a_{ij} = \phi( a_{jk}^{0}a_{ik}^{0}a_{ij}^{0}).$$ 
Suppose that only two of them are equal to $1$, say $\epsilon_{ij}=1$, $\epsilon_{ik}=1$ and $\epsilon_{jk}=0$. Then 

\begin{center}
$\phi(a_{ij}^{1}a_{ik}^{1}a_{jk}^{0}) = \tau_{i}a_{ij} \tau_{i} \tau_{i}a_{ik}\tau_{i}a_{jk} =  \tau_{i}a_{ij}a_{ik}\tau_{i}a_{jk} =  \tau_{i}a_{ij}a_{ik}a_{jk}\tau_{i} = \tau_{i}a_{jk}a_{ik}a_{ij}\tau_{i} = \tau_{i}a_{jk}\tau_{i} \tau_{i}a_{ik}\tau_{i}a_{ij} =\phi(a_{jk}^{0}a_{ik}^{1}a_{ij}^{1}),$
\end{center}
 which can be presented geomerically as Fig.~\ref{proof-lem-A3}.
\begin{figure}[h!]
\begin{center}
 \includegraphics[width =12cm]{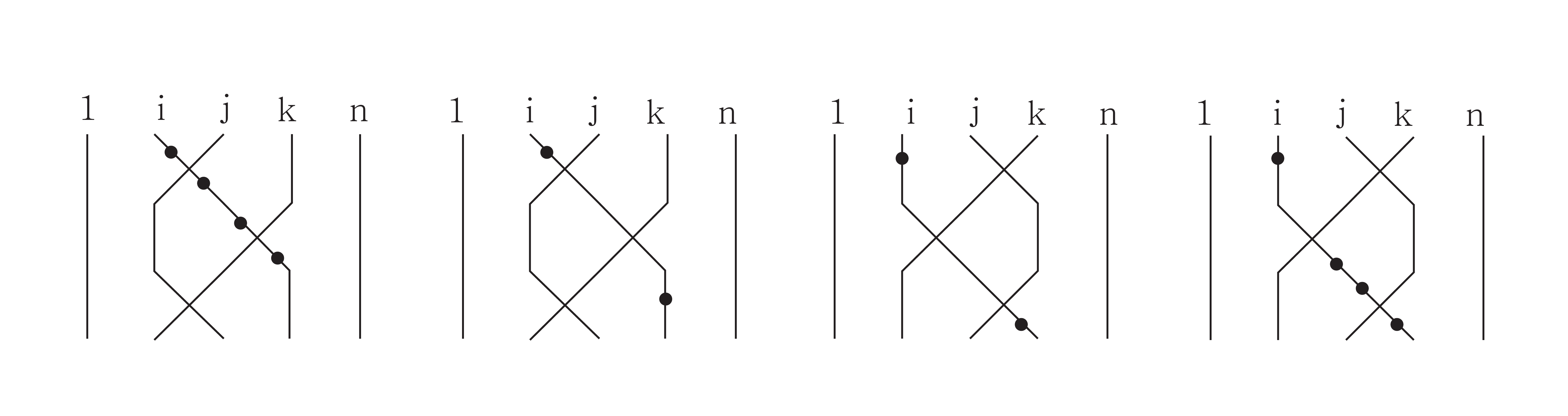}

\end{center}

 \caption{$\phi(a_{ij}^{1}a_{ik}^{1}a_{jk}^{0}) = \phi(a_{jk}^{1}a_{ik}^{1}a_{ij}^{0})$}\label{proof-lem-A3}
\end{figure}

 Therefore, $\phi$ is well-defined. 

\end{proof}

Let $H_{n,d}^{2} = \{\beta \in G_{n,d}^{2}~|~ N_{i}(\beta) = 0~\text{mod}~2, ~ i \in \{1, \cdots, n\} \}$, where $N_{i}(\beta)$ is the number of $\tau_{i}$ in $\beta$. Notice that, by the definition of $\phi$, it is clear that the image of $G_{n,p}^{2}$ by $\phi$ goes into the set $H_{n,d}^{2}$. Since every relation of $G_{n,d}^{2}$ preserves the number of $\tau$'s modulo $2$, $H_{n,d}^{2}$ is a subgroup of $G_{n,d}^{2}$. 

\begin{lem}
Let $\beta = T_{0} a_{i_{1}j_{1}}T_{2}\cdots T_{k-1}a_{i_{k}j_{k}}T_{k} \cdots T_{m-1}a_{i_{m}j_{m}}T_{m} \in H$, where $T_{k}$ is a product of $\tau$'s. Define a function $\chi : H_{n,d}^{2} \rightarrow G_{n,p}^{2}$  by 
$$ \chi(\beta) = a_{i_{1}j_{1}}^{\epsilon_{1}} \cdots a_{i_{k}j_{k}}^{\epsilon_{k}} \cdots a_{i_{m}j_{m}}^{\epsilon_{m}},$$
where $N_{i_{k}}$ is the number of $\tau_{i_{k}}$ in $T_{0}, \cdots, T_{k-1}$ and $\epsilon_{k} = N_{i_{k}} + N_{j_{k}}$. Then $\chi$ is well-defined.
\end{lem}

\begin{proof}
Consider $\beta = T_{0} a_{i_{1}j_{1}}T_{2}\cdots T_{k-1}a_{i_{k}j_{k}}T_{k} \cdots T_{m-1}a_{i_{m}j_{m}}T_{m} \in H$.
Assume that $a_{i_{k}j_{k}}$ contains in one of relations, which can be applied to $\beta$. If $a_{i_{k}j_{k}}$ is contains in the relation (6) $\tau_{i}\tau_{j}a_{ij}\tau_{j}\tau_{i} = a_{ij}$, then the number of  $\tau_{i_{k}}$ and $\tau_{j_{k}}$ in $T_{k-1}$ preserves modulo $2$ and then the sum of the numbers of $\tau_{i_{k}}$ and $\tau_{j_{k}}$ in $T_{0}, \cdots, T_{k-1}$ is not changed modulo $2$. Therefore $\epsilon_{k} = N_{i_{k}} + N_{j_{k}}$ is preserved modulo $2$. If $a_{i_{k}j_{k}}$ is contained in the relation (7) $\tau_{k}a_{ij} = a_{ij}\tau_{k}$, then the number of $\tau_{i_{k}}$ and $\tau_{j_{k}}$ in $T_{k-1}$ is preserved, since in relation (7) $i,j,k$ are different respectively. Therefore the number of $\tau_{i_{k}}$ and $\tau_{j_{k}}$ in $T_{0}, \cdots, T_{k-1}$ is not changed, and $\epsilon_{k} = N_{i_{k}} + N_{j_{k}}$ is preserved modulo $2$. The relation  (2) $a_{ij}a_{kl} = a_{kl}a_{ij}$ is preserved along $\chi$ by relation $a_{ij}^{\epsilon_{1}}a_{kl}^{\epsilon_{2}} = a_{kl}^{\epsilon_{2}}a_{ij}^{\epsilon_{1}}$ for different $i,j,k,l$. Suppose that $a_{i_{k}j_{k}}$ is appeared in the relation (1) $a_{ij}^{2}=1$. Then $i_{k} =i_{k+1}$, $j_{k} = j_{k+1}$ and there are no $\tau$'s in $T_{k}$. Therefore $\epsilon_{k}= \epsilon_{k+1}$ and the relation is preserved by  $(a_{ij}^{\epsilon})^{2}=1$. Suppose that $a_{i_{k}j_{k}}$ is appeared in relation (3) $a_{ij}a_{ik}a_{jk}=a_{jk}a_{ik}a_{ij}$. Suppose that $a_{i_{k}j_{k}}a_{i_{k+1}j_{k+1}}a_{i_{k+2}j_{k+2}}=a_{i_{k+2}j_{k+2}}a_{i_{k+1}j_{k+1}}a_{i_{k}j_{k}}$ and $i_{k} = i_{k+1}$, $j_{k} = i_{k+2}$ and $j_{k+1} = j_{k+2}$. It is sufficient to show that $\epsilon_{k}+\epsilon_{k+1}+\epsilon_{k+2} = 0$ modulo $2$. 
Then there are no $\tau$'s in $T_{k}$,$T_{k+1}$,$T_{k+2}$, and hence $N_{i_{k}} = N_{i_{k+1}}$, $N_{j_{k}} = N_{i_{k+2}}$, $N_{j_{k+1}} = N_{j_{k+2}}$. Therefore 
 \begin{center}
$\epsilon_{k}+\epsilon_{k+1}+\epsilon_{k+2} = N_{i_{k}}+N_{j_{k}}+N_{i_{k+1}}+N_{j_{k+1}}+N_{i_{k+2}}+N_{j_{k+2}} = 0 ~\text{mod}~2.$
\end{center}
and the proof is completed.

\end{proof}

\begin{thm}\label{cor_isoGH}
$G_{n,p}^{2}$ is isomorphic to $H_{n,d}^{2}$.
\end{thm}

\begin{proof}
We will show that $\chi \circ \phi = 1_{G_{n,p}^{2}}$ and $\phi \circ \chi = 1_{H_{n,d}^{2}}$. To show that $\chi \circ \phi = 1_{G_{n,p}^{2}}$ let $\beta = a_{i_{1}j_{1}}^{\epsilon_{1}}  \cdots a_{i_{m}j_{m}}^{\epsilon_{m}} \in G_{n,p}^{2}$. Then 
\begin{center}
$\chi(\phi(\beta)) = \chi(\tau_{i_{1}}^{\epsilon_{1}}a_{i_{1}j_{1}} \tau_{i_{1}}^{\epsilon_{1}} \cdots \tau_{i_{m}}^{\epsilon_{m}}a_{i_{m}j_{m}}\tau_{i_{m}}^{\epsilon_{m}}) = a_{i_{1}j_{1}}^{\theta_{1}}  \cdots a_{i_{m}j_{m}}^{\theta_{m}}$.
\end{center}
For each $k$, the number of all $\tau$'s which are appeared from $a_{i_{1}j_{1}}^{\epsilon_{1}}  \cdots a_{i_{k-1}j_{k-1}}^{\epsilon_{m}}$ is even. Since $\theta_{k}$ is equal to the number of all $\tau_{i_{k}}$ and $\tau_{j_{k}}$ modulo $2$, $\theta_{k} = \epsilon_{k}$. Therefore $\chi \circ \phi = 1_{G_{n,p}^{2}}$.

Now we will show that $\phi \circ \chi = 1_{H_{n,d}^{2}}$. Let 

$$\beta = T_{0} a_{i_{1}j_{1}}T_{1}\cdots T_{k-1}a_{i_{k}j_{k}}T_{k} \cdots T_{m-1}a_{i_{m}j_{m}}T_{m} \in H.$$ 
Firstly we will show that $\beta$ is equivalent to an element in the form 
$$\tau_{i_{1}}^{\epsilon_{1}}a_{i_{1}j_{1}} \tau_{i_{1}}^{\epsilon_{1}} \cdots \tau_{i_{m}}^{\epsilon_{m}}a_{i_{m}j_{m}}\tau_{i_{m}}^{\epsilon_{m}}.$$ 
By the relations $\tau_{i}\tau_{j} = \tau_{j}\tau_{i}$ and $\tau_{k}a_{ij} = a_{ij}\tau_{k}$, $\beta$ is equivalent to the form $$T_{0} a_{i_{1}j_{1}}T_{2}\cdots T_{k-1}a_{i_{k}j_{k}}T_{k} \cdots T_{m-1}a_{i_{m}j_{m}}T_{m}$$
such that $T_{k}$ only has $\tau_{i_{k}}$, $\tau_{j_{k}}$, $\tau_{i_{k+1}}$ and $\tau_{j_{k+1}}$. In point of $a_{i_{k}j_{k}}$, it is a product of the following; 
$$ A_{i_{k}j_{k}} = \tau_{i_{k}}^{\theta^{f}_{i_{k}}}\tau_{j_{k}}^{\theta^{f}_{j_{k}}}a_{i_{k}j_{k}}\tau_{i_{k}}^{\theta^{b}_{i_{k}}}\tau_{j_{k}}^{\theta^{b}_{j_{k}}}.$$

%By the relation $\tau_{i}\tau_{j}a_{ij}\tau_{i}\tau_{j} = a_{ij}$, if one of $\theta^{f}_{i_{k}}$ and $\theta^{b}_{i_{k}}$ is equal to $1$, then $\beta$ is equivalent to an element in the form $\theta^{f}_{j_{k}}$ and $\theta^{b}_{j_{k}}$ are equal to $0$. 

Now we claim that there is an element $\beta'$ which is a product of 
$$A'_{i_{k}j_{k}} = \tau_{i_{k}}^{\theta^{f}_{i_{k}}}\tau_{j_{k}}^{\theta^{f}_{j_{k}}}a_{i_{k}j_{k}}\tau_{i_{k}}^{\theta^{b}_{i_{k}}}\tau_{j_{k}}^{\theta^{b}_{j_{k}}},$$ 
where   $\tau_{i_{k}}^{\theta^{f}_{i_{k}}} = \tau_{i_{k}}^{\theta^{b}_{i_{k}}}$ and  $\tau_{i_{k}}^{\theta^{f}_{j_{k}}} = \tau_{j_{k}}^{\theta^{b}_{i_{k}}}$. Suppose that $i$ is an index such that $A_{ij} =  \tau_{i}^{\theta^{f}_{i}}\tau_{j}^{\theta^{f}_{j}}a_{ij}\tau_{i}^{\theta^{b}_{i}}\tau_{j}^{\theta^{b}_{j}}$ is the first part in which $\theta_{i}^{f} \neq \theta_{i}^{b}$, say $\theta_{i}^{f}=1$ and $\theta_{i}^{b}=0$. Notice that if $\theta_{i}^{b}=1$ and $\theta_{i}^{f}=1$, then $\tau_{i}^{\theta_{i}^{b}}$ can be move to the very next $A_{il}$ and it is same to the case of $\theta_{i}^{f}=1$ and $\theta_{i}^{b}=0$. Since $\beta$ is in $H$, the number of $\tau_{i}$ should be even and there is $A_{ik} = \tau_{i}^{\theta'^{f}_{i}}\tau_{k}^{\theta^{f}_{k}}a_{ik}\tau_{i}^{\theta'^{b}_{i}}\tau_{k}^{\theta^{b}_{k}}$ such that $\theta'^{f}_{i} =1 $ or $\theta'^{b}_{i} =1 $ as \ref{equ_1-1} and  \ref{equ_1-2}.

  \begin{eqnarray}
  \label{equ_1-1}\beta &=& \cdots A_{ij} \cdots A_{ik} \cdots\\
     \label{equ_1-2}   &=& \cdots  \tau_{j}^{\theta^{f}_{j}}a_{ij}\tau_{i}^{\theta^{b}_{i}}\tau_{j}^{\theta^{b}_{j}} \cdots \tau_{i}^{\theta^{f}_{i}}\tau_{k}^{\theta^{f}_{k}}a_{ik}\tau_{k}^{\theta^{b}_{k}} \cdots .
       \end{eqnarray}
       
If $\theta'^{f}_{i} =1 $ and there are no $A_{il}$, then by the relations $a_{ij}\tau_{k} = \tau_{k}a_{ij}$ and $\tau_{i}\tau_{j} = \tau_{j}\tau_{i}$, $\tau_{i}$ can be move to the next of $a_{ij}$ as~\ref{equ_2-1} and \ref{equ_2-2}.

\begin{eqnarray}
\label{equ_2-1} \beta 
        &=& \cdots  \tau_{i}\tau_{j}^{\theta^{f}_{j}}a_{ij}\tau_{j}^{\theta^{b}_{j}} \cdots A_{**} \cdots \tau_{i}\tau_{k}^{\theta^{f}_{k}}a_{ik}\tau_{k}^{\theta^{b}_{k}} \cdots\\
    \label{equ_2-2}   &=& \cdots  \tau_{i}\tau_{j}^{\theta^{f}_{j}}a_{ij}\tau_{i}\tau_{j}^{\theta^{b}_{j}} \cdots A_{**} \cdots \tau_{k}^{\theta^{f}_{k}}a_{ik}\tau_{k}^{\theta^{b}_{k}} \cdots.
       \end{eqnarray}
       
If $\theta'^{f}_{i} =1 $ and there is a $A_{il}$'s, then by the relation $\tau_{i}^{2} =1$ new $\tau$'s are appeared and by the relations $a_{ij}\tau_{k} = \tau_{k}a_{ij}$ and $\tau_{i}\tau_{j} = \tau_{j}\tau_{i}$, $\tau_{i}$'s can be move to the next of $a_{ij}$ and to the front of $a_{il}$ as~\ref{equ_3-1}, \ref{equ_3-2} and \ref{equ_3-3}.
\begin{eqnarray}
\label{equ_3-1} \beta 
        &=& \cdots  \tau_{i}\tau_{j}^{\theta^{f}_{j}}a_{ij}\tau_{j}^{\theta^{b}_{j}} \cdots A_{il} \cdots \tau_{i}\tau_{k}^{\theta^{f}_{k}}a_{ik}\tau_{k}^{\theta^{b}_{k}} \cdots\\
\label{equ_3-2}       &=&  \cdots  \tau_{i}\tau_{j}^{\theta^{f}_{j}}a_{ij}\tau_{j}^{\theta^{b}_{j}} \cdots \tau_{i}\tau_{i} A_{il} \cdots \tau_{i}\tau_{k}^{\theta^{f}_{k}}a_{ik}\tau_{k}^{\theta^{b}_{k}} \cdots\\
\label{equ_3-3}       &=&  \cdots  \tau_{i}\tau_{j}^{\theta^{f}_{j}}a_{ij}\tau_{i}\tau_{j}^{\theta^{b}_{j}} \cdots \tau_{i} A_{il}\tau_{i} \cdots \tau_{k}^{\theta^{f}_{k}}a_{ik}\tau_{k}^{\theta^{b}_{k}} \cdots.
       \end{eqnarray}

%If $\theta'^{b}_{i} =1 $ and there are no $A_{il}$, then in the same manner, new $\tau$'s are appeared and they can be moved to the next of $a_{ij}$ and to the front of $a_{ik}$ as~\ref{equ_4-1}, \ref{equ_4-2} and \ref{equ_4-3}.

%\begin{eqnarray}
%\label{equ_4-1} \beta 
     %   &=& \cdots  \tau_{i}\tau_{j}^{\theta^{f}_{j}}a_{ij}\tau_{j}^{\theta^{b}_{j}} \cdots A_{**} \cdots \tau_{k}^{\theta^{f}_{k}}a_{ik}\tau_{i}\tau_{k}^{\theta^{b}_{k}} \cdots\\
% \label{equ_4-2}      &=& \cdots  \tau_{i}\tau_{j}^{\theta^{f}_{j}}a_{ij}\tau_{j}^{\theta^{b}_{j}} \cdots A_{**} \cdots \tau_{i}\tau_{i} \tau_{k}^{\theta^{f}_{k}}a_{ik}\tau_{i}\tau_{k}^{\theta^{b}_{k}} \cdots\\
%    \label{equ_4-3}   &=& \cdots  \tau_{i}\tau_{j}^{\theta^{f}_{j}}a_{ij} \tau_{i}\tau_{j}^{\theta^{b}_{j}} \cdots A_{**} \cdots\tau_{i} \tau_{k}^{\theta^{f}_{k}}a_{ik}\tau_{i}\tau_{k}^{\theta^{b}_{k}} \cdots.
%       \end{eqnarray}
 Analgously the case $\theta'^{b}_{i} =1 $ also can be solved. 
 
 Now assume that $\beta = \prod_{k=1}^{m} = \tau_{i_{k}}^{\theta_{i_{k}}}\tau_{j_{k}}^{\theta_{j_{k}}}a_{i_{k}j_{k}}\tau_{i_{k}}^{\theta_{i_{k}}}\tau_{j_{k}}^{\theta_{j_{k}}}$. 
 %is a product of $$A_{i_{k}j_{k}} = \tau_{i_{k}}^{\theta_{i_{k}}}\tau_{j_{k}}^{\theta_{j_{k}}}a_{i_{k}j_{k}}\tau_{i_{k}}^{\theta_{i_{k}}}\tau_{j_{k}}^{\theta_{j_{k}}}.$$ 
 Then 
\begin{eqnarray*}
 \phi \circ \chi(\beta) &=& \phi \circ \chi( \prod_{k=1}^{m} = \tau_{i_{k}}^{\theta_{i_{k}}}\tau_{j_{k}}^{\theta_{j_{k}}}a_{i_{k}j_{k}}\tau_{i_{k}}^{\theta_{i_{k}}}\tau_{j_{k}}^{\theta_{j_{k}}})\\
 &=& \phi(\prod_{k=1}^{m} a_{i_{k}j_{k}}^{\theta_{i_{k}j_{k}}}) = \prod_{k=1}^{m}(\tau_{i_{k}}^{\theta_{i_{k}j_{k}}}a_{i_{k}j_{k}}\tau_{i_{k}}^{\theta_{i_{k}j_{k}}}),
\end{eqnarray*}

where $\theta_{i_{k}j_{k}} = \theta_{i_{k}}+\theta_{j_{k}}$ mod $2$. If $\theta_{i_{k}}= \theta_{j_{k}}=1$, by the relation $\tau_{i}\tau_{j}a_{ij}\tau_{j}\tau_{i} =a_{ij}$, $\beta$ and $\phi \circ \chi(\beta)$ are same elements in $G_{n,d}^{2}$. If $\theta_{i_{k}}=1$ and $ \theta_{j_{k}}=0$(or $\theta_{i_{k}}=0$ and $ \theta_{j_{k}}=1$, by the relation $\tau_{i}\tau_{j}a_{ij}\tau_{j}\tau_{i} =a_{ij}$(in other word,$ \tau_{i}a_{ij}\tau_{i}=\tau_{j}a_{ij}\tau_{j}$), $\beta$ and $\phi \circ \chi(\beta)$ are same elements in $G_{n,d}^{2}$ and hence $\phi \circ \chi = 1_{G_{n,d}^{2}}$.
\end{proof}

\section{$G_{n,d}^{2}$ and $G_{n+1}^{2}$}

In~\cite{Ma}, maps from $G_{n}^{k}$ to $G_{n-1}^{k}$ and from $G_{n}^{k}$ to $G_{n-1}^{k-1}$ are constructed. In these maps, each generator is mapped to a generator. Here we justify these maps and extended to the case of mappings from $G_{n}^{2}$ to $G_{n-1,d}^{2}$. In section 3, we obtained including maps $i : G_{n}^{2} \rightarrow G_{n,p}^{2}$ and $\phi : G_{n,p}^{2} \rightarrow G_{n,d}^{2}$. To show that they are including maps, we found right inverse of them. In this section we consider relations between $G_{n,d}^{2}$ and $G_{n+1}^{2}$.
 
Let us define $\psi$ from $G_{n+1}^{2}$ to $G_{n,d}^{2}$ by 
\begin{center}
$\psi(a_{ij}) = \left\{
\begin{array}{cc} 
     a_{ij} & \text{if}~n+1 \not\in \{i,j\}, \\
    \tau_{i}
 &  \text{if}~ j= n+1, \\
 \tau_{j} &  \text{if}~ i= n+1. \\
   \end{array}\right.$
\end{center}

Then the following lemma can be proved.
\begin{lem}\label{func(n+1)topoint}
A mapping $\psi$ from $G_{n+1}^{2}$ to $G_{n,d}^{2}$is well-defined.
\end{lem}

\begin{proof}
It is enough to show that every relation for $G_{n+1}^{2}$ is preserved by $\psi$. If every index is different with $n+1$, it is clear.
Consider a relation $a_{ij}a_{kl}=a_{kl}a_{ij}$ for distinct $i,j,k,l \in \{1, \cdots, n+1\}$. If one of $i,j,k$ and $l$ is $n+1$, say $i=n+1$, then 
$$\psi(a_{(n+1)j}a_{kl}) = \tau_{j}a_{kl}= a_{kl}\tau_{j} =\psi(a_{kl}a_{(n+1)j}).$$

 For the relations $a_{ij}a_{ik}a_{jk} = a_{jk}a_{ik}a_{ij}$, if one of $i,j$ and $k$  is $n+1$, say $i=n+1$, then 
 $$\psi( a_{ij}a_{ik}a_{jk}) = \tau_{j}\tau_{k}a_{jk} = a_{jk}\tau_{k}\tau_{j} = \psi(a_{jk}a_{ik}a_{ij}).$$
Clearly, the relation $a_{ij}^{2}=1$ is preserved along $\psi$ and the statement is proved. 
\end{proof}

Now we define inverse mapping. We define $\omega$ from $G_{n,d}^{2}$ to $G_{n+1}^{2}$ by $\omega(a_{ij}) = a_{ij}$ and $\omega(\tau_{i}) = a_{i(n+1)}$, for example, see Fig.~\ref{exa_omega}.

\begin{figure}[h!]
\begin{center}
 \includegraphics[width =10cm]{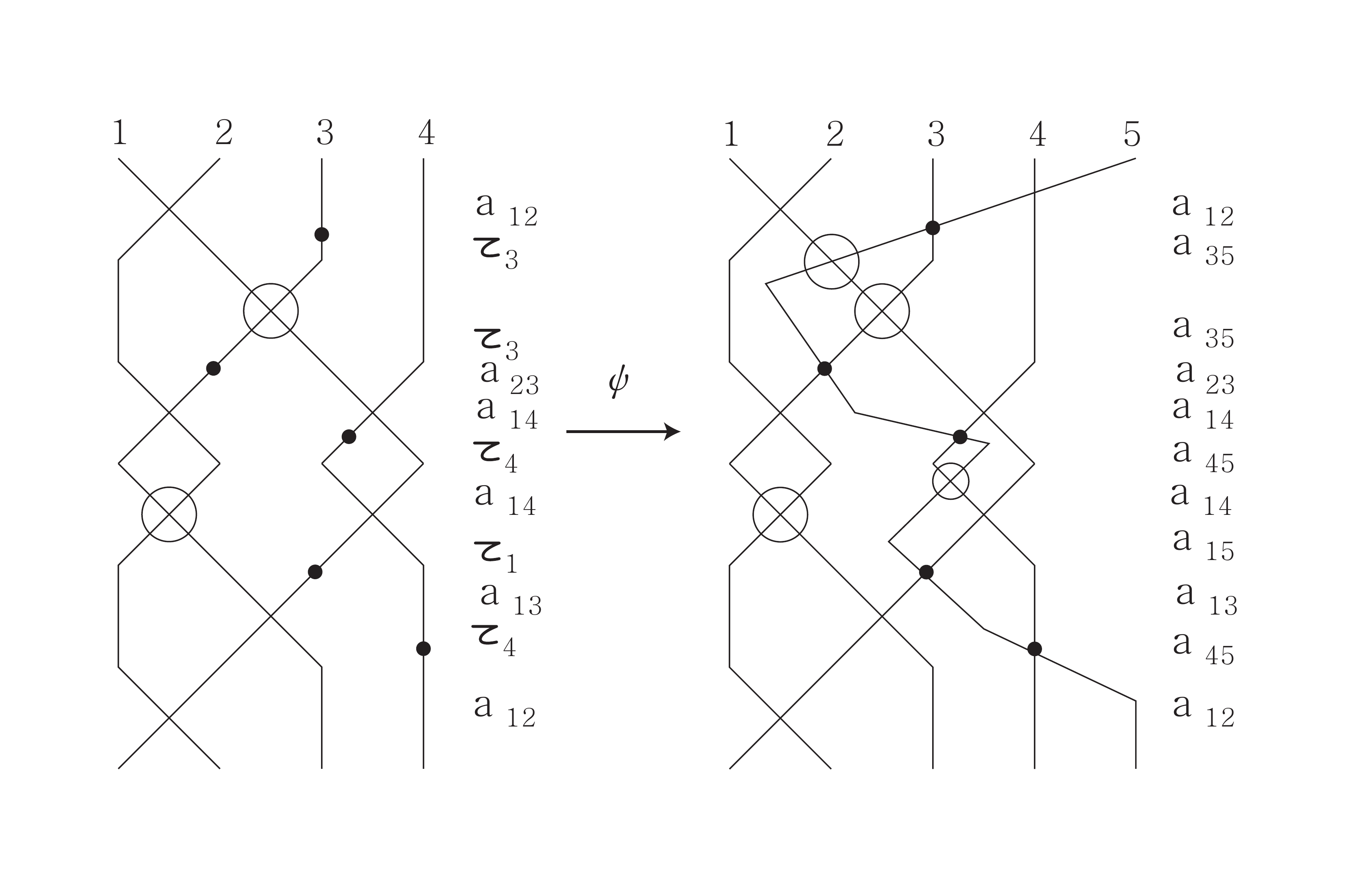}

\end{center}

 \caption{Example of a image of the mapping $\omega$}\label{exa_omega}
\end{figure}

 But this mapping is not well-defined because the relation $\tau_{i}\tau_{j} = \tau_{j}\tau_{i}$ in $G_{n,d}^{2}$ goes to $a_{i(n+1)}a_{j(n+1)}= a_{j(n+1)}a_{i(n+1)}$ but $a_{i(n+1)}$ and $a_{j(n+1)}$ are not commutative. However, we can modify the mapping and get the following lemma.
\begin{lem}
A mapping $\omega : G_{n,d}^{2} \rightarrow G_{n+1}^{2}/ \langle a_{i(n+1)}a_{j(n+1)}= a_{j(n+1)}a_{i(n+1)}\rangle$, defined by $\omega(a_{ij}) = a_{ij}$ and $\tau_{i} = a_{i(n+1)}$, is well-defined.
\end{lem}

\begin{proof}
It is sufficient to show that every relation for $G_{n,d}^{2}$ goes to trivial in $G_{n+1}^{2}/ \langle a_{i(n+1)}a_{j(n+1)}= a_{j(n+1)}a_{i(n+1)}\rangle$. For relations (1),(2),(3) and (4), it is clear. Consider relation $\tau_{i}\tau_{j} = \tau_{j}\tau_{i}$. Then
$$\omega(\tau_{i}\tau_{j}) =  a_{i(n+1)}a_{j(n+1)}= a_{j(n+1)}a_{i(n+1)} = \omega(\tau_{j}\tau_{i}).$$ 
The relation (6) $\tau_{i}\tau_{j}a_{ij}\tau_{j}\tau_{i} = a_{ij}$ can be written by $\tau_{i}\tau_{j}a_{ij} = a_{ij}\tau_{j}\tau_{i}$. Then 
$$\omega(\tau_{i}\tau_{j}a_{ij}) = a_{i(n+1)}a_{jn+2}a_{ij}=a_{ij}a_{j(n+1)}a_{i(n+1)} = \omega(a_{ij}\tau_{j}\tau_{i}).$$
 Finally, the relations (5) and (7) are preserved by the relation $a_{i(n+1)}a_{j(n+1)}= a_{j(n+1)}a_{i(n+1)}$ and the proof is completed.
%$\omega(a_{ij}\tau_{k}) = a_{ij}a_{kn+1} = a_{kn+1}a_{ij} = \omega(\tau_{k}a_{ij})$. 
\end{proof}

Geometrically the relations $a_{i(n+1)}a_{j(n+1)} = a_{j(n+1)}a_{i(n+1)}$ mean that two consecutive classical crossings pass a virtual crossing, see Fig.~\ref{forbidden_comp}. Generally, this move in this case,  is called {\it a forbidden move by $(n+1)$-th strand.}

\begin{figure}[h!]
\begin{center}
 \includegraphics[width =8cm]{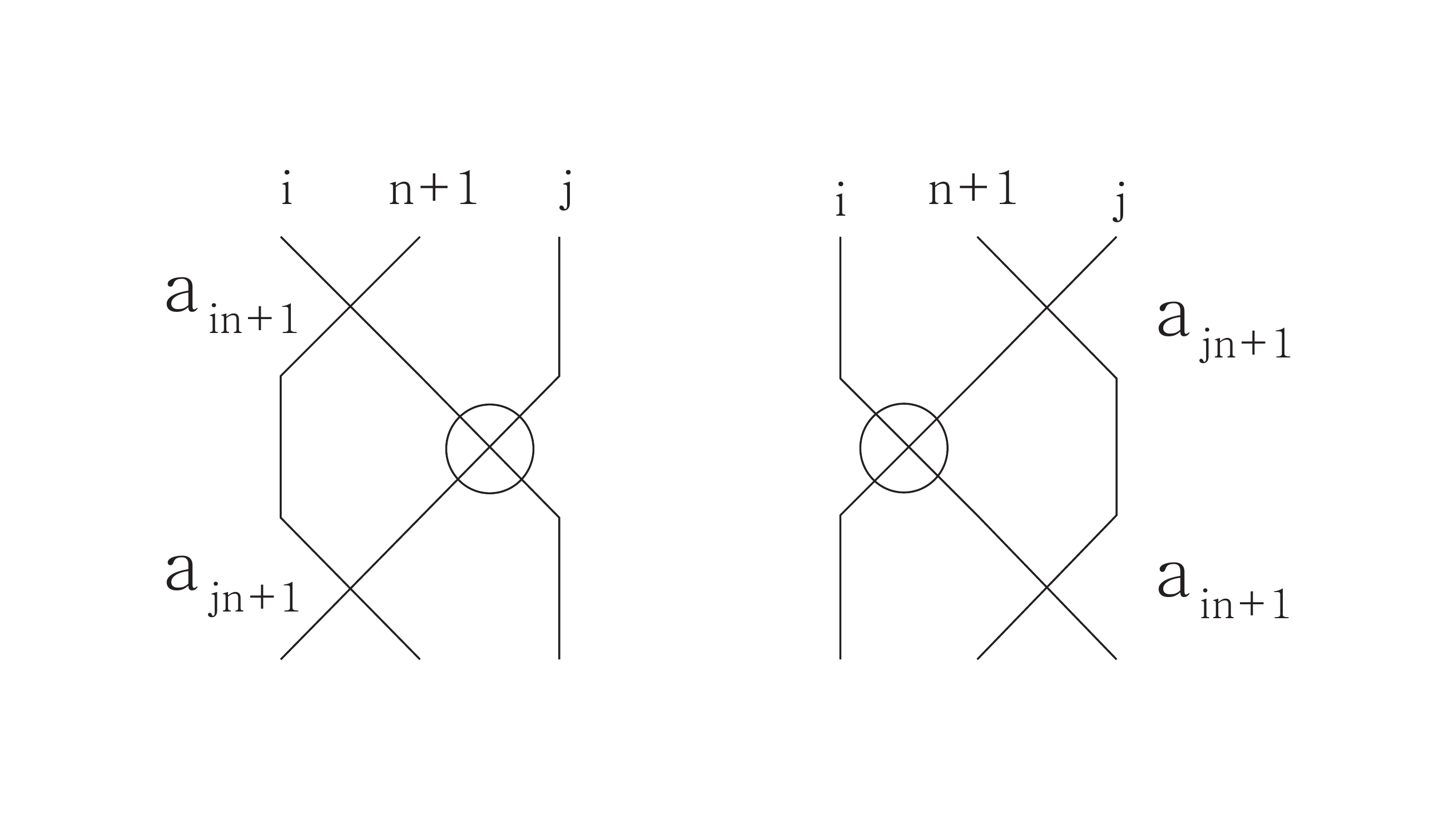}

\end{center}

 \caption{forbidden moves}\label{forbidden_comp}
\end{figure}

\begin{cor}
The groups $G_{n,d}^{2}$ and $G_{n+1}^{2}/ \langle a_{i(n+1)}a_{j(n+1)}= a_{j(n+1)}a_{i(n+1)}\rangle$ are isomorphic.
\end{cor}

\begin{proof}
We will show that $\omega \circ \psi = 1_{G_{n,d}^{2}}$ and $\psi \circ \omega = 1_{G_{n+1}^{2}/ \langle a_{i(n+1)}a_{j(n+1)}= a_{j(n+1)}a_{i(n+1)}\rangle}$. For generators $a_{ij}$ in $G_{n+1}^{2}/ \langle a_{i(n+1)}a_{j(n+1)}= a_{j(n+1)}a_{i(n+1)}\rangle$, 
\begin{center}
$\psi(a_{ij}) =  \left\{
\begin{array}{cc} 
     a_{ij} & \text{if}~n+1 \in \{i,j\}, \\
    \tau_{i}
 &  \text{if}~ j= n+1, \\
 \tau_{j} &  \text{if}~ i= n+1. \\
   \end{array}\right.$
\end{center}
By definition of $\omega$, $\omega(\psi(a_{ij}))= a_{ij}$. Hence  $\omega \circ \psi = 1_{G_{n,d}^{2}}$. Analogously, it can be shown that $\psi \circ \omega = 1_{G_{n+1}^{2}/ \langle a_{i(n+1)}a_{j(n+1)}= a_{j(n+1)}a_{i(n+1)}\rangle}$.
\end{proof}

We proved that there are isomorphism from $G_{n,p}^{2}$ to $H_{n,d}^{2} \subset G_{n,d}^{2}$(Theorem~\ref{theorem1}) and isomorphism from $G_{n,d}^{2}$ to $G_{n+1}^{2}/  \langle a_{i(n+1)}a_{j(n+1)}= a_{j(n+1)}a_{i(n+1)}\rangle$ (Theorem~\ref{theorem2}). From those theorems, it is followed that there is a monomorphism $\omega \circ \phi $ from $G_{n,p}^{2}$ to $G_{n+1}^{2}/  \langle a_{i(n+1)}a_{j(n+1)}= a_{j(n+1)}a_{i(n+1)}\rangle$. It means that every braid of $n$ strands with parity can be presented by braid with $(n+1)$ strands up to forbidden moves by $(n+1)$-th strand. That is, braids with parity are in itself.
On the other hand, parity of braid has geometrical meaning: how $(n+1)$-th strand is knotted with area, which is surrounded by other strans, see Fig.~\ref{9_exa_parityton+1}. In Fig.~\ref{9_exa_parityton+1},
the image of $a_{ij}^{1}$ by $\omega \circ \phi $ $(n+1)$-th strand passes area in black, which is surrounded by $i$-th and $j$-th strands. Notice that the image of $a_{ij}^{1}$ by $\omega \circ \phi $ any relation in $G_{n+1}^{2}$  cannot be applied. That is, linking between $(n+1)$-th strand and the area in black cannot be disappeared. 

Conversely, for a braid $\beta$ with $(n+1)$ strands such that the number of crossings between $i$-th and $(n+1)$-th strands is even for every index $i \neq n+1$, that is, a braid of $n$ strands with parity(or with points) can be obtained from $\beta \in \psi^{-1}(H_{n,d}^{2})$ by deleting one strand.
\begin{figure}[h!]
\begin{center}
 \includegraphics[width =10cm]{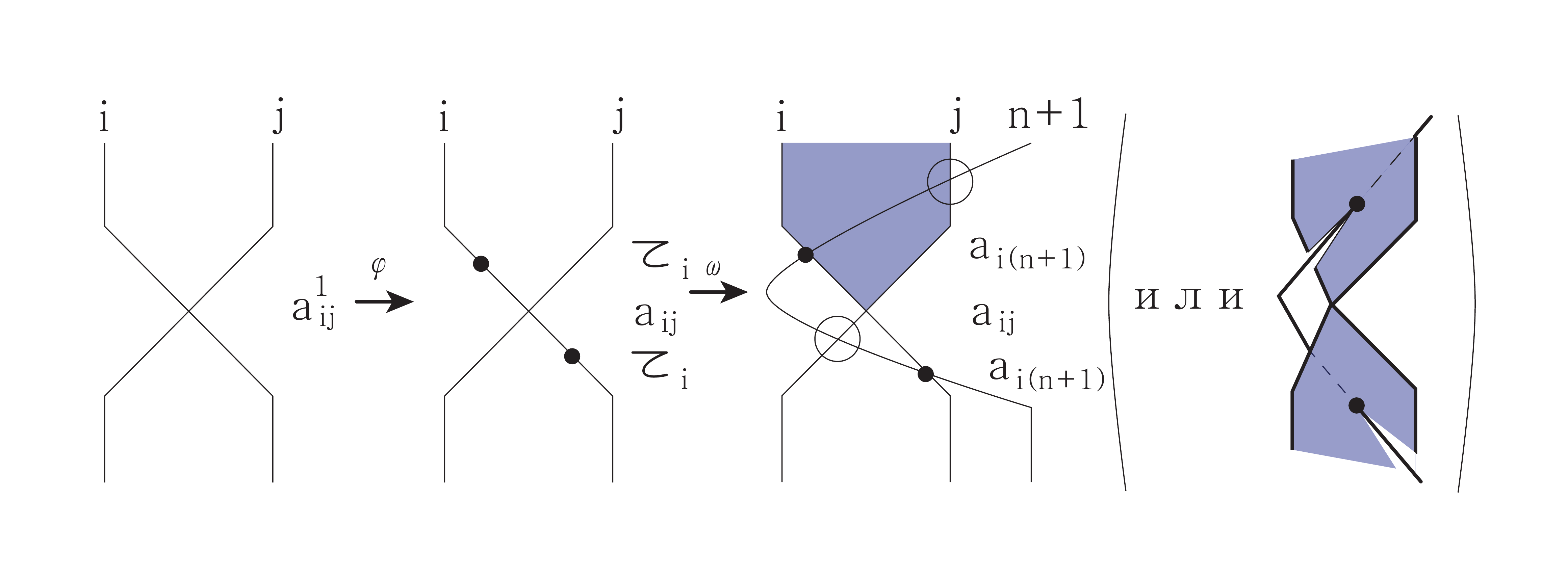}

\end{center}

 \caption{The image of $a_{ij}^{1}$ by $\omega \circ \phi$ }\label{9_exa_parityton+1}
\end{figure}
Now define a mapping $\psi_{m} : G_{n+1}^{2} \rightarrow G_{n,d}^{2}$ by
\begin{center}
$\psi_{m}(a_{ij})  = \left\{
\begin{array}{cc} % brackets may be (...), [...], \{...\}, or left out
    \tau_{i} & \text{if}~ j= n+1,~ i<m, \\
     \tau_{i-1} & \text{if}~ j= n+1, ~i>m. \\
      a_{ij} & \text{if}~ i,j<m, ~i,j \neq n+1 \\
       a_{i(j-1)} & \text{if}~ i<m, j>m, ~i,j \neq n+1\\
       a_{(i-1)(j-1)} & \text{if}~ i, j>m,~ i,j \neq n+1\\
   \end{array}\right.$
   \end{center}
Analogously $\psi_{m}$ is well-defined(see lemma~\ref{func(n+1)topoint}). Then, for each braid $\beta$ of $(n+1)$ strands such that the number of crossings between $i$-th and $m$-the strands is even for each index $i \neq m$, that is, $\beta \in \psi_{m}^{-1}(H_{n,d}^{2})$, and a braid with parity can be obtained.
\begin{exa}
Let  $\beta = a_{12}a_{23}a_{13}a_{23}a_{13}a_{23}a_{12}a_{23}$ in $G_{3}^{2}$. Note that $\psi_{m}(\beta) \in H_{2,d}^{2}$ for each index $m$. Then
\begin{itemize}
\item $\chi(\psi_{1}(\beta))=\chi(\tau_{1}a_{12}\tau_{2}a_{12}\tau_{2}a_{12}\tau_{1}a_{12}) = a_{12}^{1}a_{12}^{0}a_{12}^{1}a_{12}^{0} \neq 1,$
\item $\chi(\psi_{2}(\beta)) =\chi(\tau_{1}\tau_{2}a_{12}\tau_{2}a_{12}\tau_{2}\tau_{1}\tau_{2})=\chi(a_{12}\tau_{1}a_{12}\tau_{1})= a_{12}^{0}a_{12}^{1}  \neq 1 ,$
\item $\chi(\psi_{3}(\beta)) =\chi(a_{12}\tau_{2}\tau_{1}\tau_{2}\tau_{1}\tau_{2}a_{12}\tau_{2})= \chi(a_{12}\tau_{2}a_{12}\tau_{2})= a_{12}^{0}a_{12}^{1}  \neq 1 .$
\end{itemize}
Since $\psi_{m}$ is well defined, $\beta$ is not trivial in $G_{3}^{2}$.
\end{exa}

As the above example, for given brad $\beta$ in $\psi_{m}^{-1}(H_{n,d}^{2}) \subset G_{n+1}^{2}$, if $\chi(\psi_{m}(\beta)) \neq 1$, then $\beta$ is not trivial and $m$-th strand of $\beta$ is important to be non-trivial for $\beta$. And the following question is followed: \\
{\it Assume that $\beta$ in $\psi_{m}^{-1}(H_{n,d}^{2}) \subset G_{n+1}^{2}$ for each index $m$. If $\beta$ it non-trivial in $G_{n+1}^{2}$, then is there index $m$ such that $\chi(\psi_{m}(\beta)) \neq 1$?}

\end{document}